\documentclass[12pt]{article}

\usepackage{tikz}
\usepackage{subfigure}
\usepackage[english]{babel}

\usepackage[center]{caption2}
\usepackage{amsfonts,amssymb,amsmath,latexsym,amsthm}
\usepackage{multirow}

\newtheorem{thm}{Theorem}[section]

\newtheorem{prob}[thm]{Problem}
\newtheorem{claim}{Claim}

\newtheorem{obs}{Observation}
\newtheorem{cons}{Construction}

\begin{document}

\title{Exact minimum codegree thresholds  for $K_4^-$-covering and $K_5^-$-covering\thanks{The work was supported by NNSF of China (No. 11671376),  NSF of Anhui Province (No. 1708085MA18), and Anhui Initiative in Quantum Information Technologies (AHY150200).}}
\author{Lei Yu$^a$, \quad Xinmin Hou$^b$,\quad Yue Ma$^c$, \quad Boyuan Liu$^d$\\
\small $^{a,b,c,d}$ Key Laboratory of Wu Wen-Tsun Mathematics\\
\small School of Mathematical Sciences\\
\small University of Science and Technology of China\\
\small Hefei, Anhui 230026, China.
}

\date{}
\maketitle

\begin{abstract}
Given two $3$-graphs $F$ and $H$, an $F$-covering of $H$ is a collection of copies of $F$ in $H$ such that each vertex of $H$ is contained in at least one copy of them. Let {$c_2(n,F)$} be the maximum integer $t$ such that every 3-graph with minimum codegree greater than $t$ has an $F$-covering.  In this note, we answer an open problem of Falgas-Ravry and Zhao (SIAM J. Discrete Math., 2016) by determining the exact value of {$c_2(n, K_4^-)$} and {$c_2(n, K_5^-)$}, where $K_t^-$ is the complete $3$-graph on $t$ vertices with one edge removed.
\end{abstract}

\section{Introduction}


Given a set $V$ and a positive integer $k$, let $\binom{V}{k}$ be the collection of $k$-element subset of $V$.
A {\em simple $k$-uniform hypergraph} (or  $k$-graph for short) $H=(V,E)$ consists of a vertex set $V$ and an edge set $E\subseteq\binom{V}{k}$.
We write graph for $2$-graph for short. For a set $S\subseteq V(H)$, the {\em neighbourhood} $N_H(S)$ of $S$ is $\{T\subseteq V(H)\backslash S:T\cup S\in E(H)\}$ and the {\em degree} of $S$ is $d_H(S)=|N_H(S)|$.
The {\em minimum $s$-degree} of $H$, denoted by $\delta_s(H)$, is the minimum $d_H(S)$ taken over all $s$-element sets of $V(H)$, and $\delta_{k-1}(H)$ and $\delta_1(H)$ are usually called the {\em minimum codegree} and the {\em minimum degree} of $H$, respectively.
An $r$-graph $H$ is called an {\em $r$-partite $r$-graph} if the vertex set of $H$ can be partitioned into $r$ parts such that each edge of $H$ intersects each part exactly one vertex. Given disjoint sets $V_1, V_2, \cdots, V_r$, let $K(V_1, V_2, \ldots, V_r)$ be the complete $r$-partite $r$-graph with vertex classes $V_1, V_2,\ldots, V_r$.

Given a $k$-graph $F$,  we say a $k$-graph $H$ has an {\em $F$-covering} if each vertex of $H$ is contained in some copy of $F$. For $0\leq i<k$, define
$$c_i(n,F)=\max\{\delta_i(H) : H \text{ is a $k$-graph on $n$ vertices with no } F\text{-covering}\}.$$
We call $c_{k-1}(n,F)$ the {\em minimum codegree threshold} for $F$-covering.


For graphs $F$, the $F$-covering problem was solved asymptotically in~\cite{Z-PhD16} by showing that $c_1(n,F)=(\frac{\chi(F)-2}{\chi(F)-1}+o(1))n$, where $\chi(F)$ is the chromatic number of $F$. For general $k$-graphs, the function {$c_i(n, F)$} was determined for some special family of $k$-graphs $F$. For example, Han, Lo, and Sanhueza-Matamala~\cite{HLS-ENDM17} proved that $c_{k-1}(n,C_s^{(k,k-1)})\leq (\frac{1}{2}+o(1))n$ for {$k\geq3, s\geq 2k^2$}, where $C_s^{(k,\ell)}$ $(1\leq \ell<k)$ is the $k$-graph on $s$ vertices such that its vertices can be ordered cyclicly so that every edge consists of $k$ consecutive vertices under this order and two consecutive edges intersect in exactly $\ell$ vertices. Han, Zang, and Zhao showed in~\cite{HZZ-JCTA17} that $c_1(n,K)=(6-4\sqrt{2}+o(1))\binom{n}{2}$, where $K$ is a complete $3$-partite $3$-graph with at least two vertices in each part.
In this note, we focus on the minimum codegree threshold for covering $3$-graphs. Let $K_t$ denote the complete $3$-graph on $t$ vertices and let $K_t^-$ denote the $3$-graph obtained from $K_t$ by removing one edge.
Falgas-Ravry and Zhao~\cite{FZ-SIAM16} determined the exact value of $c_2(n,K_4)$ for $n>98$ and gave lower and upper bounds of $c_2(n, K_4^-)$ and $c_2(n, K_5^-)$. More specifically, they proved the following theorem.

\begin{thm}[Theorem 1.2 in~\cite{FZ-SIAM16}]\label{ASY K_4^-}
Suppose $n=6m+r$ for some $r\in \{0,1,2,3,4,5\}$ and $m\in \mathbb{N}$ with $n\geq 7$. Then
\begin{equation*}
c_2(n,K_4^-)=
\begin{cases}
2m-1\text{ or }2m&\text{ if } r=0,\\
2m&\text{ if } r\in\{1,2\},\\
2m\text{ or }2m+1&\text{ if } r\in\{3,4\},\\
2m+1&\text{ if } r=5.\\
\end{cases}
\end{equation*}
\end{thm}

\begin{thm}[Theorem 1.4 in~\cite{FZ-SIAM16}]\label{ASY K_5^-}
$\lfloor\frac{2n-5}{3}\rfloor\leq c_2(n,K_5^-)\leq \lfloor\frac{2n-2}{3}\rfloor$.
\end{thm}

Falgas-Ravry and Zhao~\cite{FZ-SIAM16} also conjectured that the gap between the upper and lower bounds for $c_2(n,K_4^-)$ could
be closed and left this as an open problem.

\begin{prob}[\cite{FZ-SIAM16}]\label{PROB: p1}
Determine the exact {value of} $c_2(n,K_4^-)$ in the case $n\equiv 0,3,4 \pmod 6$.
\end{prob}

In this note, we  determine not only the exact value of $c_2(n,K_4^-)$  but also the exact value of $c_2(n, K_5^-)$.

\begin{thm}\label{exact K_4^-}
$c_2(n,K_4^-)=\lfloor\frac{n}{3}\rfloor$.
\end{thm}

The theorem solved Problem~\ref{PROB: p1} completely.

\begin{thm}\label{exact K_5^-}
$c_2(n,K_5^-)=\lfloor\frac{2n-2}{3}\rfloor$.
\end{thm}

This result completes Theorem~\ref{ASY K_5^-}.




The following are some definitions and notation used in our proofs.
For a $k$-graph $H$ and $x\in V(H)$, the {\em link graph}  of $x$, denoted by $H(x)$,  is the $(k-1)$-graph with {vertex set} $V(H)\setminus\{x\}$ and edge set $N_H(x)$.
Given a graph $G$ and a positive integer vector ${\bf k}\in Z^{V(G)}_+$, the {\it ${\bf k}$-blowup} of $G$, denoted by $G^{(\bf k)}$, is the graph obtained by replacing every vertex $v$ of $G$ with ${\bf k}(v)$ different vertices where a copy of $u$ is adjacent to a copy of $v$ in the blowup graph if and only if $u$ is adjacent to $v$ in $G$. We call the collection of copies of $v\in V(G)$ in $G^{(\bf k)}$ the blowup of $v$.
When there is no confusion, we write $ab$ and $abc$ as a shorthand for $\{a,b\}$ and $\{a,b,c\}$, respectively.
Given two $r$-graphs $H$ and $F$, we say $H$ is $F$-free if $H$ contains no subgraph isomorphic to $F$. Given a positive integer $n$, write $[n]$ for the set $\{1,2, \ldots, n\}$.

In the rest of the note, we give proofs of Theorems~\ref{exact K_4^-} and~\ref{exact K_5^-}.


\section{Proof of Theorems~\ref{exact K_4^-} and~\ref{exact K_5^-} }
We will construct extremal 3-graphs for $K_4^-$ and $K_5^-$ with minimum codegree matching the upper bounds in Theorems~\ref{ASY K_4^-} and~\ref{ASY K_5^-}, respectively.

\subsection{Proof of Theorem~\ref{exact K_4^-}}

We first give an observation, which can be verified directly from the definitions.
\begin{obs}\label{OBS: o1}
Let $H$ be a $3$-graph and $x\in V(H)$. If $H(x)$
is triangle-free and the subgraph of $H(x)$ induced by an edge $e\in E(H)$ with $x\notin e$ is $P_2$-free,
then $x$ can not be covered by a $K_4^-$ in $H$, where $P_2$ is a path of length two.
\end{obs}

By Theorem~\ref{ASY K_4^-}, to show Theorem~\ref{exact K_4^-}, it is sufficient to construct  3-graphs $H$ on $n$ vertices for $n\equiv 0, 3,4 \pmod 6$ and with  $\delta_2(H)=\lfloor \frac n3\rfloor$ such that $H$ has no $K_4^-$-covering. We distinguish the proof into three cases.
Let $C_6$ be the 6-cycle $v_1v_2v_3v_4v_5v_6v_1$.

\vspace{5pt}
\noindent{\bf Construction A:} Let $G_1$ be the graph obtained from $C_6$ and the 5-cycle $123451$ by adding the edges $1v_1, 1v_3, 2v_2, 2v_5, 3v_4, 3v_6, 4v_3, 4v_5, 5v_2, 5v_6$.

\vspace{5pt}
\noindent{\bf Construction B:} Let $G_2$ be the graph obtained from $C_6$ and the 8-cycle $123456781$ by adding the edges $1v_1, 1v_3, 2v_2, 2v_6, 3v_1, 3v_5, 4v_3, 4v_6, 5v_2, 5v_4, 6v_3, 6v_5, 7v_4, 7v_6, 8v_2, 8v_5$.

\vspace{5pt}
\noindent{\bf Construction C:} Let $G_3$  be the graph obtained from $C_6$ and the 8-cycle $123456781$ by adding a new vertex 9 and the edges $19, 39,79$, $1v_1,1v_3,2v_2,2v_6,3v_1,3v_4,4v_3,4v_5,5v_4$, $5v_6, 6v_1,6v_5,7v_3,7v_6,8v_2,8v_4, 9v_2, 9v_5$.

It can be checked that  $G_1, G_2, G_3$ are triangle-free; therefore, so are the blowups of them.

\vspace{5pt}
\noindent{\bf Case $1$.} $n=6m$ for some integer $m\ge 1$.
\vspace{5pt}

Define a positive integer vector ${\bf k}_1\in Z_+^{V(G_1)}$ by ${\bf k}_1(v_i)=m-1$ for $i\in [6]$ and ${\bf k}_1(i)=1$ for $i\in [5]$.

\begin{cons}\label{Construction A:} Let $V_1, \ldots, V_6$ be six disjoint sets of the same size $m-1$ and let $x$ be a specific vertex. Define the 3-graph $H_1$ on vertex set $\{x\}\cup [5]\cup (\cup_{i=1}^6V_i)$ such that the following holds:
\begin{itemize}
\item[(1)]
The link graph of $x$, $H_1(x)$, consists of the ${\bf k}_1$-blowup of $G_1$ by replacing $v_i$ by $V_i$ for $i\in [6]$ and a perfect matching between $V_1$ and $V_4$.

\item[(2)]
A triple $abc\in E(H_1)$ if $x\notin\{a,b,c\}$ and the subgraph induced by $\{a,b,c\}$ in ${H_1}(x)$ is $P_2$-free.

\end{itemize}
\end{cons}


\begin{claim}\label{CLAIM: c1}
$H_1$ contains no $K_4^-$-covering and $\delta_2(H_1)=2m=\lfloor\frac n3\rfloor$.
\end{claim}

\noindent{\bf Proof of Claim 1. }
By the definition of $G_1$, $v_1$ and $v_4$ have no common neighbor. So   by (1) of Construction~\ref{Construction A:}, ${H_1}(x)$ is triangle-free. By (2) of Construction~\ref{Construction A:}, any two incident edges of ${H_1}(x)$ are not contained in one edge of $H_1$. By Observation~\ref{OBS: o1}, $x$ is contained in no copy of $K_4^-$ in $H_1$. So $H_1$ has no  $K_4^-$-covering. 

 By (1) of Construction~\ref{Construction A:},  one can check that ${H_1}(x)$ is $2m$-regular. So  $d_{H_1}(x,a)=2m$ for all $a\in V\setminus\{x\}$.
Now we consider the degree of the pair $\{a,b\}$  with $x\notin\{a,b\}$.
If  $ab\in E({H_1}(x))$, then by (2) of Construction~\ref{Construction A:}, $N_{H_1}(x,a)\cap N_{H_1}(a,b)=\emptyset$, $N_{H_1}(x,b)\cap N_{H_1}(a,b)=\emptyset$ and $N_{H_1}(x,a)\cap N_{H_1}(x,b)=\emptyset$; or equivalently, for any $c\notin N_{H_1}(x,a)\cup N_{H_1}(x,b)$, $\{a,b,c\}$ forms an edge of $H_1$.
So $d_{H_1}(a,b)=6m-2\times 2m=2m$.
If $ab\notin E({H_1}(x))$ then $x\notin N_{H_1}(a,b)$. By (2) of the construction of $H_1$, $N_{H_1}(x,a)\cap N_{H_1}(x,b)\cap N_{H_1}(a,b)=\emptyset$; or equivalently, for any $c\notin (N_{H_1}(x,a)\cap N_{H_1}(x,b))\cup\{x,a,b\}$, we have $abc\in E(H_1)$.
So $d_{H_1}(a,b)=6m-3-|N_{H_1}(a,x)\cap N_{H_1}(b,x)|\ge 4m-3\ge 2m$ if $m>1$. If $m=1$, then $H_1(x)$ is the 5-cycle $123451$, one can check that $d_{H_1}(a,b)\ge 2=2m$.

Case 1 follows directly from Claim~\ref{CLAIM: c1}.

\vspace{5pt}
\noindent{\bf Case 2: } $n=6m+3$ for some integer $m\ge 1$.

Define a positive integer vector ${\bf k}_2\in Z_+^{V(G_2)}$ by ${\bf k}_2(v_i)=m-1$ for $i\in [6]$ and ${\bf k}_2(i)=1$ for $i\in [8]$.

\begin{cons}\label{Construction B:}
 Let $V_1, \ldots, V_6$ be six disjoint sets of the same size $m-1$ and let $x$ be a specific vertex. Define the 3-graph $H_2$ on vertex set $\{x\}\cup [8]\cup (\cup_{i=1}^6V_i)$ such that the following holds:
\begin{itemize}
\item[(1)]
The link graph of $x$, $H_2(x)$, consists of the ${\bf k}_2$-blowup of $G_2$ by replacing $v_i$ with $V_i$ for $1\le i\le 6$,  a perfect matching between $V_1$ and $V_4$ and a matching $\{15, 26, 37, 48\}$.

\item[(2)]
A triple $abc\in E(H_2)$ if $x\notin\{a,b,c\}$ and the subgraph induced by $\{a,b,c\}$ in ${H_2}(x)$ is $P_2$-free.

\end{itemize}
\end{cons}



\begin{claim}\label{CLAIM: c2}
$H_2$ contains no $K_4^-$-covering and $\delta_2(H_2)=2m+1=\lfloor\frac n3\rfloor$.
\end{claim}

\noindent{\bf Proof of Claim 2.} By the definition of $G_2$, $N_{G_2}(v_1)\cap N_{G_2}(v_4)=\emptyset$ and $N_{G_2}(1)\cap N_{G_2}(5)=N_{G_2}(2)\cap N_{G_2}(6)=N_{G_2}(3)\cap N_{G_2}(7)=N_{G_2}(4)\cap N_{G_2}(8)=\emptyset$.
 So by (1) of Construction~\ref{Construction B:}, ${H_2}(x)$ is triangle-free, too; and by (2) of Construction~\ref{Construction B:},  any two incident edges of $H_2(x)$ are not contained in one edge of $H_2$. By Observation~\ref{OBS: o1}, $x$ is contained in no copy of $K_4^-$ in $H_2$. So $H_2$ has no  $K_4^-$-covering. 

By (1) of Construction~\ref{Construction B:},  ${H_2}(x)$ is $(2m+1)$-regular. So  $d_{H_2}(x,a)=2m+1$ for all $a\in V(H_2)\setminus\{x\}$. Now assume $\{a,b\}\subseteq V(H_2)\setminus\{x\}$. If  $ab\in E({H_2}(x))$, then by (2) of Construction~\ref{Construction B:}, $N_{H_2}(x,a)\cap N_{H_2}(a,b)=\emptyset$, $N_{H_2}(x,b)\cap N_{H_2}(a,b)=\emptyset$ and $N_{H_2}(x,a)\cap N_{H_2}(x,b)=\emptyset$; or equivalently, for any $c\notin N_{H_2}(x,a)\cup N_{H_2}(x,b)$, $\{a,b,c\}$ forms an edge of $H_2$.
So $d_{H_2}(a,b)=6m+3-2 (2m+1)=2m+1$.
If $ab\notin E({H_2}(x))$ then $x\notin N_{H_2}(a,b)$. By (2) of the construction of $H_2$, $N_{H_2}(x,a)\cap N_{H_2}(x,b)\cap N_{H_2}(a,b)=\emptyset$; or equivalently, for any $c\notin (N_{H_2}(x,a)\cap N_{H_2}(x,b))\cup\{x,a,b\}$, $abc\in E(H_2)$.
So we have $d_{H_2}(a,b)=6m+3-3-|N_{H_2}(a,x)\cap N_{H_2}(b,x)|\ge 4m-1\ge  {2m+1}$.



Case 2 follows from Claim~\ref{CLAIM: c2}.

\vspace{5pt}
\noindent{\bf Case $3$:} $n=6m+4$ for some integer $m\ge 1$.

Define a positive integer vector ${\bf k}_3\in Z_+^{V(G_3)}$ by ${\bf k}_3(v_i)=m-1$ for $i\in [6]$ and ${\bf k}_3(i)=1$ for $i\in [9]$.

\begin{cons}\label{Construction C:}
 Let $V_1, \ldots, V_6$ be six disjoint sets of the same size $m-1$ and let $x$ be a specific vertex. Define a 3-graph $H_3$ on vertex set $\{x\}\cup [9]\cup (\cup_{i=1}^6V_i)$ such that the following holds:
\begin{itemize}
\item[(1)]
The link graph of $x$, $H_3(x)$, consists of the ${\bf k}_3$-blowup of $G_3$ and
a matching $\{15, 26, 48\}$.

\item[(2)]
A triple $abc\in E(H_3)$ if $x\notin \{a,b,c\}$ and the subgraph induced by  $\{a,b,c\}$ in ${H_3}(x)$ is $P_2$-free.

\end{itemize}
\end{cons}



\begin{claim}\label{CLAIM: c3}
$H_3$ contains no $K_4^-$-covering and $\delta_2(H_3)=2m+1=\lfloor\frac n3\rfloor$.
\end{claim}

\noindent{\bf Proof of Claim~\ref{CLAIM: c3}:}
 By (1) of Construction~\ref{Construction C:}, one can check that $H_3(x)$ is triangle-free; and by (2) of Construction~\ref{Construction C:},  any two incident edges of $H_3(x)$ are not contained in one edge of $H_3$. By Observation~\ref{OBS: o1}, $x$ is contained in no copy of $K_4^-$ in $H_3$. So $H_3$ has no  $K_4^-$-covering. 

By the construction of $H_3(x)$, one can check that $H_3(x)$ is almost $(2m+1)$-regular, i.e. $d_{H_3(x)}(a)=2m+1$ for all vertices $a\in V(H_3)\setminus\{x,1\}$ and $d_{{H_3}(x)}(1)=2m+2$.  So  $d_{H_3}(x,a)=2m+1$ for all $a\in V(H_3)\setminus\{x,1\}$ and $d_{H_3}(x,1)=2m+2$. Now assume $\{a,b\}\subseteq V(H_3)\setminus\{x\}$.
 If  $ab\in E(H_3(x))$, by (2) of Construction~\ref{Construction C:}, $N_{H_3}(x,a)\cap N_{H_3}(a,b)=\emptyset$, $N_{H_3}(x,b)\cap N_{H_3}(a,b)=\emptyset$, and for any $c\in V(H_3)\setminus(N_{H_3}(x,a)\cup N_{H_3}(x,b))$, $\{a,b,c\}$ forms an edge of $H_3$. Since $H_3(x)$ is triangle-free, $N_{H_3}(x,a)\cap N_{H_3}(x,b)=\emptyset$. If $1\notin \{a,b\}$ then $d_{H_3}(a,b)=|V(H_3)|-|N_{H_3}(x,a)|-| N_{H_3}(x,b)|=6m+4-2(2m+1)=2m+2$. Now assume $1\in\{a,b\}$, say $a=1$. Then $d_{H_3}(1,b)=|V(H_3)|-|N_{H_3}(x,1)|-| N_{H_3}(x,b)|=6m+4-(2m+2)-(2m+1)=2m+1$.
If $ab\notin E(H_3(x))$ then $x\notin N_{H_3}(a,b)$. By (2) of the construction of $H_3$, $N_{H_3}(x,a)\cap N_{H_3}(x,b)\cap N_{H_3}(a,b)=\emptyset$; or equivalently, for any $c\notin (N_{H_3}(x,a)\cap N_{H_3}(x,b))\cup\{x,a,b\}$, $abc\in E(H_3)$.
So we have $d_{H_3}(a,b)=6m+4-3-|N_{H_3}(a,x)\cap N_{H_3}(b,x)|\ge 4m\ge 2m+1$.



Case 3 follows from Claim~\ref{CLAIM: c3}.




Theorem~\ref{exact K_4^-} follows from Cases 1,2,3 and Theorem~\ref{ASY K_4^-}.



\subsection{Proof of Theorem~\ref{exact K_5^-}}

The following theorem is well known in graph theory.

\begin{thm}[K\"{o}nig~\cite{Konig16}]\label{Konig}
Let $G$ be a bipartite graph with maximum
degree $\Delta$. Then $E(G)$ can be partitioned into $M_1,  M_2, \ldots, M_{\Delta}$ so
that each $M_i\  (1\le i\le \Delta)$ is a matching in $G$.
In particular, if $G$ is $\Delta$-regular then $E(G)$ can be partitioned into $\Delta$ perfect matchings.
\end{thm}


\begin{cons}\label{cons D} Given positive integers $m, \ell$ with $m\le \ell$ and two disjoint sets $V_1, V_2$ with $|V_1|\le |V_2|=m$, by Theorem~\ref{Konig}, the edge set of the complete bipartite graph $K(V_1, V_2)$ has a partition $M_1, M_2, \ldots, M_m$ such that each $M_i$ $(1\le i\le m)$ is a matching.
Let $T$ be the 3-partite 3-graph with vertex classes $V_1\cup V_2\cup [\ell]$ and edge set $$E(H)=\bigcup_{i=1}^m\{ e\cup\{i\} : e\in M_i\}.$$
\end{cons}

\begin{proof}[Proof of Theorem~\ref{exact K_5^-}]

We first give the extremal 3-graph for $K_5^-$.




\begin{cons}\label{Cons E}
Given a positive integer $m$ and three disjoint sets $V_1, V_2, V_3$ such that  $m-1\le |V_1|\le |V_2|=m\le |V_3|\le m+1$ and $|V_3|-|V_1|\le 1$, then $3m-1\le \sum\limits_{i=1}^3|V_i|\le 3m+1$.  Denote $V_3=[\ell]$. Then $\ell=m$ or $m+1$. Let $T$ be the 3-partite 3-graph on vertex set $V_1\cup V_2\cup V_3$ defined by Construction~\ref{cons D}. Let $x$ be {a specific} vertex not in $V_1\cup V_2\cup V_3$. Define the 3-graph $H_4$ on vertex set $V_1\cup V_2\cup V_3\cup\{x\}$ such that the following holds.

(1) The link graph of $x$, $H_4(x)$, consists of three complete bipartite graphs $K(V_1, V_2)$, $K(V_1, V_3)$ and $K(V_2, V_3)$.

(2) Each triple $e\notin E(K(V_1, V_2, V_3))$ is an edge of $H_4$.

(3) $E(T)\subseteq E(H_4)$.

\end{cons}


Let $n=|V(H_4)|$. Then $3m\le n\le 3m+2$.

\begin{claim}\label{Claim: 1.4}
 $H_4$ has no $K_5^-$-covering and $\delta_2(H)\ge  \lfloor\frac{3n-2}{3}\rfloor$.
\end{claim}
\noindent{\bf Proof of Claim~\ref{Claim: 1.4}:}
 We show that $x$ is contained in no copy of $K_5^-$ in $H_4$. Suppose to the contrary that $H_4$ contains a copy of $K_5^-$, say $K$, covering $x$. Denote $V(K)=\{x,a,b,c,d\}$. Then there is at least one part $V_i$ $(1\le i\le 3)$ such that $|V_i\cap\{a,b,c,d\}|\ge 2$. By (1) of Construction~\ref{Cons E}, there is no edge of $H_4$ included in $\{x\}\cup V_i$. So at least one edge  connecting $x$ and $\{a,b,c,d\}$ misses from $K$. As $K$ is a copy of $K_5^-$, there is exact one edge  between $x$ and $\{a,b,c,d\}$ missed and so $\{a,b,c,d\}$ induces a copy of the complete 3-graph $K_4$ in $H_4$.  From Construction~\ref{cons D},  {$\Delta_2(T)\le 1$}.
 By (2) and (3) of Construction~\ref{Cons E}, each pair of vertices chosen from different parts of $V_1, V_2, V_3$ has at most one neighbor in the remaining part.
 Thus a putative $K_4$ induced by $\{a,b,c,d\}$ intersects at most two parts of $V_1, V_2, V_3$. But this is impossible. In fact, if there is some $1\le i\le 3$ such that $a,b,c\in V_i$, then $xab, xac\notin E(H_4)$, a contradiction. So assume $a,b\in V_i$ and $c,d \in V_j\ (i\neq j)$ for some $i,j\in\{1,2,3\}$. Then $xab,xcd\notin E(H_4)$, a contradiction too.



Now we compute the minimum codegree of $H_4$. Choose two distinct vertices $a,b\in V(H_4)$. If $x\in\{a,b\}$, assume $x=a$ and $b\in V_i$, then by (1) of Construction~\ref{Cons E},
 $$d(x,b)=n-1-|V_i|\geq n-1-\left\lceil\frac{n-1}{3}\right\rceil=\left\lfloor\frac{2n-2}{3}\right\rfloor.$$
If $a,b\in V_i$ for some $1\le i\le 3$ then, by (2) of Construction~\ref{Cons E}, $d(a,b)=\sum\limits_{i=1}^3|V_i|-2=n-3\geq\left\lfloor\frac{2n-2}{3}\right\rfloor$.
If $a\in V_i, b\in V_j$ $(i\not=j)$, then $$d(a,b)=|V_i|+|V_j|-2+1+d_{T}(a,b)\geq\left\lfloor\frac{2n-2}{3}\right\rfloor,$$
where the inequality holds since {$d_T(a,b)=1$ when $\{i, j\}=\{1,2\}$} or {$\{i,j\}\subseteq \{1,2,3\}$ and $|V_1|=|V_2|=|V_3|=m$}.

This completes the proof of Claim~\ref{Claim: 1.4}.

By Claim~\ref{Claim: 1.4}, we have
 $$c_2(n, K_5^-)\geq \delta_2(H_4)=\left\lfloor\frac{2n-2}{3}\right\rfloor.$$
By Theorem~\ref{ASY K_5^-}, we have Theorem~\ref{exact K_5^-}.

 \end{proof}

\end{document}